\documentclass[11pt]{article}
\usepackage{amsthm, amsmath, amssymb, amsfonts, url, booktabs, tikz, setspace, fancyhdr, enumerate}
\usepackage[margin = 1in]{geometry}

\usepackage[czech,english]{babel} 

\usepackage{hyperref}

\usepackage{cleveref}

\definecolor{refkey}{gray}{.75}
\definecolor{labelkey}{gray}{.2}

\usepackage[textsize=scriptsize,colorinlistoftodos]{todonotes}

\newtheorem{theorem}{Theorem}[section]
\newtheorem{proposition}[theorem]{Proposition}
\newtheorem{lemma}[theorem]{Lemma}
\newtheorem{corollary}[theorem]{Corollary}

\theoremstyle{definition}

\theoremstyle{remark}

\newcommand{\overbar}[1]{\mkern 1.5mu\overline{\mkern-1.5mu#1\mkern-1.5mu}\mkern 1.5mu}

\newcommand{\EE}{\mathbb{E}}
\newcommand{\PP}{\mathbb{P}}

\newcommand{\E}{\mathcal{E}}

\newcommand{\mockalph}[1]{}

\tikzstyle{p}+=[fill=black, circle, minimum width = 1pt, inner sep =
1pt]
\tikzstyle{w}+=[fill=white, draw, circle, minimum width = 1pt, inner sep =
1.5pt]

\title{\vspace{-0.7cm}Extended commonality of paths and cycles via Schur convexity}
\author{Jang Soo Kim\thanks{Department of Mathematics, 
		Sungkyunkwan University. 
		Email: {\tt jangsookim@skku.edu}.}
\and
Joonkyung Lee\thanks{Department of Mathematics, 
		Hanyang University, Seoul
		and Extremal Combinatorics and Probability Group, Institute for Basic Sciences (IBS). 
		Email: \texttt{joonkyunglee}@\texttt{hanyang.ac.kr}.
		}
}
\date{}

\begin{document}
\maketitle

\begin{abstract}
  A graph $H$ is \emph{common} if the number of monochromatic copies of $H$ in a 2-edge-colouring of the complete graph $K_n$ is asymptotically minimised by the random colouring, or equivalently, $t_H(W)+t_H(1-W)\geq 2^{1-e(H)}$ holds for every graphon $W:[0,1]^2\rightarrow [0,1]$, where $t_H(.)$ denotes the homomorphism density of the graph $H$. Paths and cycles being common is one of the earliest cornerstones in extremal graph theory, due to Mulholland and Smith (1959), Goodman (1959), and Sidorenko (1989).

We prove a graph homomorphism inequality that extends the commonality of paths and cycles.
Namely, $t_H(W)+t_H(1-W)\geq t_{K_2}(W)^{e(H)} +t_{K_2}(1-W)^{e(H)}$ whenever $H$ is a path or a cycle and $W:[0,1]^2\rightarrow\mathbb{R}$ is a bounded symmetric measurable function.

This answers a question of Sidorenko from 1989, who proved a slightly weaker result for even-length paths to prove the commonality of odd cycles. 
Furthermore, it also settles a recent conjecture of Behague, Morrison, and Noel in a strong form, who asked if the inequality holds for graphons $W$ and odd cycles $H$. Our proof uses Schur convexity of complete homogeneous symmetric functions, which may be of independent interest.
\end{abstract}

\section{Introduction}
Given a bounded measurable symmetric function $W:[0,1]^2\rightarrow\mathbb{R}$ and a graph $H$, 
let 
\begin{align*}
    t_H(W) :=\int_{[0,1]^{V(H)}} \prod_{ij\in E(H)} W(x_i,x_j)
    \prod_{i\in V(H)} d x_i ,
\end{align*}
where the integration is taken with respect to the Lebesgue measure. This functional $t_H(.)$ is often called the \emph{(weighted) homomorphism density of $H$}, which generalises normalised homomorphism counts from $H$ to another graph $G$. 

Various results in extremal graph theory can be interpreted by using homomorphism densities, especially by using \emph{graphons} $W$, i.e., measurable
symmetric functions $W:[0,1]^2\rightarrow [0,1]$,
although extensions to general real-valued functions~\cite{CL16,H10} or even to complex-valued functions~\cite{H09,LS21} are certainly possible. 
We refer the reader to the modern theory of graph limits~\cite{L12} for more examples.

One of the central concepts that can be rephrased conveniently by using homomorphism densities is the \emph{commonality} of graphs. A graph $H$ is \emph{common} if the number of monochromatic $H$-copies in a 2-edge-colouring of the complete graph $K_n$ is asymptotically minimised by the random colouring.
The modern language rewrites the commonality of $H$ as the simple inequality
\begin{align*}
    t_H(W)+t_H(1-W) \geq 2^{1-e(H)}
\end{align*}
for every graphon $W$, where $e(H)$ denotes the number of edges in $H$.

Since Goodman's formula~\cite{G59} and the famous conjectures of Erd\H{o}s~\cite{E62common} and of Burr--Rosta~\cite{BR80}, later disproved by Thomason~\cite{Thom89} and by Sidorenko~\cite{Sid89common}, respectively, common graphs have been extensively studied
\cite{F08,GLLV22,HHKNR12,JST96,KVW22,Sid96}. 
Amongst many, perhaps the most fundamental examples of common graphs are paths and cycles.
Our main result is to prove a new homomorphism density inequality for paths and cycles, which extends their commonality. 
For brevity, a bounded symmetric measurable function $W:[0,1]^2\rightarrow \mathbb{R}$ is said to be a \emph{kernel}.

\begin{theorem}\label{thm:main}
Let $H$ be a path or a cycle and let $W$ be a kernel.
Then 
\begin{align}\label{eq:main}
    t_{H}(W) +t_{H}(1-W) \geq t_{K_2}(W)^{e(H)} +t_{K_2}(1-W)^{e(H)}.
\end{align}
\end{theorem}

As an immediate consequence,~\Cref{thm:main} extends the commonality of paths and cycles to kernels.
\begin{corollary}\label{cor:path}
Let $H$ be a path or a cycle and let $W$ be a kernel.
Then
\begin{align}\label{eq:common}
    t_{H}(W) +t_{H}(1-W) \geq 2^{1-e(H)}.
\end{align}
\end{corollary}
As \( t_{K_2}(1-W) = 1- t_{K_2}(W) \), 
the substitution $t_{K_2}(W)=x+1/2$ gives $t_{K_2}(W)^{m} +t_{K_2}(1-W)^{m}
=(1/2+x)^m +(1/2-x)^m$. For $m>1$, this polynomial attains its global minimum at $x=0$ and hence, the corollary follows.
In~\cite{Sid89common,Sid96}, Sidorenko proved \Cref{cor:path} for cycles, even-length paths, and paths of length~$m=r2^t + 1$, $r\leq 9$, but left the general odd-length paths case as a question.
\Cref{cor:path} thus completes the result of Sidorenko and answer his question in the affirmative. 
Furthermore, the proof technique allows us to obtain stability results for both~\Cref{thm:main} and~\Cref{cor:path}; see~\Cref{sec:stability} for more details. 

\medskip

\Cref{thm:main} can also be interpreted as a `convexity-type' homomorphism inequality, as the proof uses convexity of certain functions and deduction of commonality from it also uses convexity. The inequality \eqref{eq:common} for kernels was even called `convexity' by Sidorenko~\cite{Sid89common,Sid96}.
More generally, both local and global convexity of the functional $t_H(.)$ has been extremely useful in proving various graph homomorphism inequalities including instances for Sidorenko's conjecture~\cite{CL20,Sid92}, commonality of graphs~\cite{GLLV22}, graph norms~\cite{LSch21}, and density increment argument for the celebrated regularity lemma~\cite{Sz78}. Hence,~\Cref{thm:main} adds a new example to the encyclopedia of fundamental homomorphism inequalities.

In particular, when $H$ is the $m$-edge path $P_m$, \eqref{eq:main} can be seen as a partial extension of the so-called Blakley--Roy
inequality~\cite{BR65}, also obtained by Mulholland and Smith~\cite{MS59} and by
London~\cite{L66}, which proves $t_{P_m}(W)\geq t_{K_2}(W)^{m}$ for 
every graphon $W$. 
In fact, it is impossible to fully extend the Blakley--Roy inequality to kernels $W$,
as $t_{P_m}(-W)=-t_{P_m}(W)$ for odd $m$.
For grahons~$W$, even stronger generalisations are known; see, for example,~\cite{BR20path}.

For cycles $H$, \Cref{thm:main} settles a conjecture of Behague, Morrison, and Noel~\cite[Conjecture~9.7]{BMN22}, which states that the inequality~\eqref{eq:main} holds for all odd cycles $H$ and graphons $W$. 
They proposed the conjecture as a natural extension of the commonality of cycles and proved it for the $5$-cycle $H$.
We remark that some cases of the Behague--Morrison--Noel conjecture or the inequality~\eqref{eq:main} for kernels and some cycles have been well-known for decades, although the conjecture appeared only very recently. 
For example, the smallest case when $H$ is a triangle is essentially Goodman's formula~\cite{G59} and the case when $H$ is an even cycle follows from the fact that even cycles are \emph{norming}, observed by Chung, Graham, and Wilson~\cite{CGW89} and later rephrased by Hatami~\cite{H10}.

\medskip

Our proof uses Schur convexity of complete homogeneous symmetric functions. 
The study of complete homogeneous symmetric functions is a central area in algebraic combinatorics, although their Schur convexity received attention only recently~\cite{AT21,TaoSchur}.
On the other hand, in extremal graph theory, the theory of symmetric functions rarely appears to be useful, to
the best of our knowledge. Our method therefore bridges between the seemingly distant areas
in a novel way, which may be of independent interest.

\section{Preliminaries}\label{sec:prelim}

When considering kernels $U$ and $W$, the notation $U=W$ always means the equality holds almost everywhere. We suppress the expression `almost everywhere' in what follows for brevity.

\medskip

Denote by $\E^+(H)$ the set of all subgraphs $F$ of $H$ on $V(H)$ with positive even number of edges. 
For a kernel $W$, let $U:=2W-1$. Then $U$ is again a kernel. By the standard multilinear expansion of $t_H(1+U)$ and $t_H(1-U)$,
\begin{align}
 \notag   t_H(W) +t_H(1-W)& = 2^{-e(H)}\Big(t_H(1+U)+t_H(1-U)\Big)\\ 
  \label{eq:tH}
   & = 2^{1-e(H)}\left(1+\sum_{F\in\E^+(H)} t_F(U)\right).
\end{align}
Analogously, one can also expand $t_{K_2}(.)$ to obtain
\begin{align}
\notag  t_{K_2}(W)^{e(H)}+t_{K_2}(1-W)^{e(H)}
  &= 2^{-e(H)} \left( t_{K_2}(1+U)^{e(H)}+t_{K_2}(1-U)^{e(H)} \right)\\
  \notag    &=2^{1-e(H)}\sum_{k=0}^{\lfloor e(H)/2\rfloor} \binom{e(H)}{2k} t_{K_2}(U)^k \\
  \label{eq:tK}
    &=2^{1-e(H)}\left(1+\sum_{F\in\E^+(H)} t_{K_2}(U)^{e(F)}\right).
\end{align}
Thus, we obtain the following statement equivalent to~\Cref{thm:main}.
\begin{proposition}
Let $H$ be a path or a cycle and let $U$ be a kernel.
Then
\begin{align}\label{eq:nonneg}
    \sum_{F\in\E^+(H)} \big(t_F(U)-t_{K_2}(U)^{e(F)}\big)\geq 0.
\end{align}
\end{proposition}

For an integer $d>0$, let $\E_{2d}^+(H)$ be the set of even subgraphs with exactly $2d$ edges. Then
\begin{align}\label{eq:Ed}
\sum_{F\in\E^+(H)} (t_F(U) - t_{K_2}(U)^{e(F)}) = \sum_{d=1}^{\lfloor e(H)/2 \rfloor} \sum_{F\in\E_{2d}^+(H)} (t_F(U) - t_{K_2}(U)^{2d}).    
\end{align}
When $H$ is a path, we shall prove~\eqref{eq:nonneg} directly by showing that $\sum_{F\in\E_{2d}^+(H)} (t_F(U) - t_{K_2}(U)^{2d})\geq 0$ for each $d=1,2,\dots,\lfloor e(H)/2 \rfloor$.

Now consider the case $H=C_m$, a cycle of length $m$.
Suppose first that $e(H)=m$ is odd. Then each $F\in\E^+(C_m)$ must be a proper subgraph; consider each $F\in\E^+(C_m)$ as a subgraph of $C_m\setminus e$ for every choice of $e\in E(H)\setminus E(F)$. 
By doing so, each $F$ counts exactly $e(H)-e(F)$ times. Thus,
\begin{align}\notag
    \sum_{F\in\E^+(H)} (t_F(U) - t_{K_2}(U)^{e(F)}) &=
    \sum_{e\in E(H)}\sum_{F\in \E^+(H\setminus e)}\frac{1}{e(H)-e(F)}(t_F(U) - t_{K_2}(U)^{e(F)})\\ \label{eq:E+}
    &=
    \sum_{e\in E(H)}\sum_{d=1}^{ \lfloor(e(H)-1)/2\rfloor }\frac{1}{e(H)-2d} \sum_{F\in\E_{2d}^+(H\setminus e)} (t_F(U) - t_{K_2}(U)^{2d}).
\end{align}
If $e(H)=m$ is even, then one extra term $(t_H(U)-t_{K_2}(U)^{e(H)})$ adds to~\eqref{eq:E+}.

If $H$ is a cycle of length $m+1$, then $H\setminus e$ is always a path of length $m$. 
Therefore, the following theorem, which will be shown in the next section, implies~\Cref{thm:main} for odd cycles $H$.
\begin{theorem}\label{thm:main_ineq}
Let $U$ be a kernel. Then for all integers $m$ and $d$ with $1\leq d\leq m/2$, 
\begin{align*}
  \sum_{F\in\E_{2d}^+(P_m)} t_F(U) \geq \binom{m}{2d} t_{K_2}(U)^{2d}.
\end{align*}
\end{theorem}
For even cycles $H$, we need an extra inequality 
\begin{align}\label{eq:even_cycle}
    t_H(U)\geq t_{K_2}(U)^{e(H)}
\end{align}
for each kernel $U$ to deduce~\Cref{thm:main}.
This is reminiscent of Sidorenko's conjecture, which states that~\eqref{eq:even_cycle} holds for every bipartite graph $H$ and every graphon $U$.
Even cycles are well-known to satisfy Sidorenko's conjecture~\cite{Sid93}, but~\eqref{eq:even_cycle} for kernels $U$ is slightly stronger than this fact.
Even so, it is not hard to verify it and a short proof will be given at the end of this section.
In fact, the inequality~\eqref{eq:even_cycle} for kernels $U$ is well-known since Chung, Graham, and Wilson's quasirandomness characterisation~\cite{CGW89}; also see~\cite{H10} for its modern interpretation in terms of graph limits. 

\medskip

We shall use some spectral properties of kernels. Following~\cite[Section~7.5]{L12}, a kernel $U$ can be seen as a Hilbert--Schmidt operator
\begin{align*}
    (Uf)(x) := \int_0^1 U(x,y)f(y) dy,
\end{align*}
 on $L^2[0,1]$. This operator then has countable real eigenvalues $(\lambda_i)_{i=1}^{\infty}$, where $|\lambda_i|\geq |\lambda_j|$ whenever $i<j$.
 Let $f_i$ be the orthonormal eigenfunction corresponding to nonzero $\lambda_i$, i.e.,
$\langle f_i,f_j \rangle = \delta_{i,j}$ and $Uf_i = \lambda_i f_i$.
Then $U$ admits the spectral decomposition $U(x,y)=\sum_{i=1}^{\infty}\lambda_i f_i(x)f_i(y)$.
 Hence,
 \begin{align*}
     t_{P_m}(U) = \sum_{i=1}^{\infty} \lambda_i^m \left(\int_0^1 f_i(x) dx\right)^2
 \end{align*}
 and moreover, by the Parseval identity, 
 \begin{align*}
    \sum_{i=1}^{\infty}\left(\int_0^1 f_i(x)dx\right)^2=
    \left\Vert\sum_{i=1}^{\infty} \left(\int_0^1 f_i(x)dx\right) f_i \right\Vert_2^2 =
     \left\Vert\sum_{i=1}^{\infty} \langle f_i,1\rangle f_i\right\Vert_2^2 \leq \|1\|_2^2 = 1,
 \end{align*}
 which was also observed in~\cite[(13)]{KVW22}. The inequality above becomes an equality if and only if
 the constant function $1$ can be expressed as a linear combination of $f_i$'s, i.e.,
 $1=\sum_{i\geq 1}\langle f_i,1\rangle f_i$.
 Let $p_i := (\int_0^1 f_i(x) dx)^2$ for each $i\geq 1$, $p_0:=1-\sum_{i\geq 1}p_i$, and $\lambda_0:=0$. 
 Then 
 for each integer $m\geq 0$, $t_{P_m}(U) = \sum_{i\geq 0} p_i\lambda_i^{m}$.
 This rephrases as
 \begin{lemma}\label{lem:spec}
 Let $U$ be a kernel.
 Then there exists a discrete random variable $X_U$ such that $\PP[X_U=\lambda_i]=p_i$,
 $i=0,1,\dots$, and hence, $t_{P_m}(U)=\EE[X_U^m]$.
 \end{lemma}
 We remark that a `discrete' analogue of this lemma was already observed by Erd\H{o}s and Simonovits~\cite[Theorem~4]{ES82_compact}.
 The spectral technique is also useful in proving the inequality~\eqref{eq:even_cycle} for even cycles $H$ and a kernel~$U$.
 Indeed, as $t_{C_{2m}}(U)=\sum_i \lambda_i^{2m}$
 \cite[(7.22)]{L12},
 \begin{align*}
    t_{C_{2m}}(U)^{1/2m} \geq |\lambda_1| \geq \Big|\sum_{i\geq 0}p_i\lambda_i\Big| = |t_{K_2}(U)|,
 \end{align*}
 which is~\eqref{eq:even_cycle} for even cycles $H$.
 Thus,~\Cref{thm:main_ineq} implies~\Cref{thm:main} for even cycles $H$ too, although the result is already known due to the fact that even cycles are \emph{norming}. For more discussions about the norming property, we refer the reader to~\cite[Chapter~14]{L12}.

\section{Proof of the main theorem}

Our goal in this section is to prove~\Cref{thm:main_ineq}, which implies~\Cref{thm:main}.
For a kernel $U$, let $q_{m,d}(U)$ denote 
the left-hand side  of the inequality in \Cref{thm:main_ineq}, i.e.,
\begin{align*}
    q_{m,d}(U) := \sum_{F\in\E_{2d}^+(P_m)} t_F(U).
\end{align*}
Let us first
have a look at a small example that illustrates what $q_{m,d}(U)$ is. If $d=1$,
then the corresponding $\E_{2d}^+(P_m)$ consists of the subgraphs of $P_m$ with two
edges and $m+1$ vertices. That is, either a 2-edge path or a matching of size two plus isolated vertices. Hence,
\begin{align*}
    q_{m,1}(U) = (m-1)\EE[X_U^2] + \binom{m-1}{2}\EE[X_U]^2,
\end{align*}
where $X_U$ is defined in~\Cref{lem:spec}. This can be rewritten as
\begin{align*}
    q_{m,1}(U) = \EE \left[\sum_{i=1}^{m-1}X_i^2+\sum_{1\leq i<j\leq m-1}X_iX_j\right],
\end{align*}
where $X_i$'s are i.i.d.~copies of $X_U$.

\medskip

Let $h_d(x_1,\dots,x_k)$ be the \emph{$k$-variable complete homogeneous symmetric function of degree $d$}. That is,
\begin{align*}
    h_d(x_1,\dots,x_k) = \sum x_1^{\ell_1}\cdots x_k^{\ell_k},
\end{align*}
where the sum is taken over all the nonnegative integer solutions of $\ell_1+\dots+\ell_k=d$.
For example, $h_2(x_1,\dots,x_k)=\sum_{i=1}^{k}x_i^2+\sum_{1\leq i<j\leq k}x_ix_j$ and hence, $q_{m,1}(U)=\EE[h_2(X_1,\dots,X_{m-1})]$.

\medskip

By generalising this observation, we express $q_{m,d}(U)$ in
terms of the expectation of the homogeneous polynomials
$h_{2d}(X_1,\dots,X_k)$ of degree $2d$, where $X_i$ is an i.i.d.~copy of
$X_U$ in~\Cref{lem:spec}. 

\begin{lemma}\label{lem:rv_interpret}
Let $U$ be a kernel. For all integers $m$ and $d$ with $1\leq d\leq m/2$,
\begin{align*}
    q_{m,d}(U) = \EE\big[h_{2d}(X_1,\dots,X_{m-2d+1})\big],
\end{align*}
where $X_i$'s are i.i.d.~copies of $X_U$ given in~\Cref{lem:spec}. 
\end{lemma}
\begin{proof}
  Let $F\in\E_{2d}^+(P_m)$. Recall that $V(F)=V(P_m)$. Enumerate 
  the $m-2d$ edges in $E(P_m)\setminus E(F)$ by
  $e_1,\dots,e_{m-2d}$ from left to right in the $m$-edge path $P_m$. Let $\ell_i$
  be the number of edges in the component of $F$ that contains the leftmost
  vertex of $e_i$. In particular,
  if the left-intersecting component to $e_i$ is an isolated vertex, then
  $\ell_i=0$.
   Denote by $\ell_{m-2d+1}$ the number of edges in the component of $F$ that contains the rightmost vertex of $P_m$.
Clearly,
  $\sum_{i=1}^{m-2d+1}\ell_i = 2d$. 
  
  Conversely, every nonnegative integer
  solution to the equation $\sum_{i=1}^{m-2d+1}\ell_i = 2d$ uniquely determines the corresponding
  $F\in\E_{2d}^+(P_m)$, which satisfies
\begin{align*}
    t_F(U) = \prod_{i=1}^{m-2d+1}t_{P_{\ell_i}}(U) = \prod_{i=1}^{m-2d+1}\EE[X_U^{\ell_i}]=\EE\left[\prod_{i=1}^{m-2d+1}X_i^{\ell_i}\right],
\end{align*}
where $X_i$'s are i.i.d.~copies of $X_U$.
Therefore, 
\begin{align*}
    q_{m,d} (U)= \sum_{F\in\E_{2d}^{+}(P_m)} t_F(U) = \sum\EE\left[\prod_{i=1}^{m-2d+1}X_i^{\ell_i}\right]= \EE\left[\sum\prod_{i=1}^{m-2d+1}X_i^{\ell_i}\right],
\end{align*}
where the last two sums are taken over all nonnegative integers $\ell_i$'s such
that $\ell_1+\dots+\ell_{m-2d+1}=2d$. Thus,
$q_{m,d}(U)=\EE[h_{2d}(X_1,\dots,X_{m-2d+1})]$. 
\end{proof}
Note that $h_{2d}(x_1,\dots,x_k)=\binom{k+2d-1}{2d} x^{2d}$ if $x_i=x$ for all $i$. Letting $x=\EE[X_U]=t_{K_2}(U)$  and $k=m-2d+1$ then gives $h_{2d}(\EE[X_1],\dots,\EE[X_{m-2d+1}])=\binom{m}{2d}t_{K_2}(U)^{2d}$, which, together with \Cref{lem:rv_interpret}, suggests that some convexity of $h_{2d}$ may prove~\Cref{thm:main_ineq}.

To formalise this idea,
we need an easy consequence of \emph{Schur convexity} of~$h_{2d}$.
A real $k$-tuple $(x_1,\dots,x_k)$ \emph{majorises} another $k$-tuple $(y_1,\dots,y_k)$ if $\sum_{i=1}^{j}x_i\geq \sum_{i=1}^{j}y_i$ for every $j=1,\dots,k$ with equality for $j=k$.
A $k$-variable real polynomial $h$ is \emph{Schur convex} if $h(x_1,\dots,x_k)\geq h(y_1,\dots,y_k)$ whenever $(x_1,\dots,x_k)$ majorises $(y_1,\dots,y_k)$.
One can deduce from the classical Schur--Ostrowski theorem~\cite[Chapter~3, A.4.~Theorem]{Marshall2011}
that $h_{2d}(x_1,\dots,x_k)$ is Schur convex (see, e.g.,~\cite{TaoSchur}), whose immediate consequence is the following lemma. 
For self-containedness, we give a brief probabilistic proof which essentially rephrases 
Barvinok's argument \cite[Lemma~3.1]{Barvinok2005} 
and pushes it slightly further; see also \cite[Remark~6.4]{AT21} and
an anonymous comment in~\cite{TaoSchur}.

\begin{lemma}\label{lem:Schur}
Let $d,k>0$ be integers. Then 
for all real numbers $x_1,\dots,x_k$,
\[
     h_{2d}(x_1,\dots,x_k)\geq
    h_{2d}(\overbrace{\overbar{x},\dots,\overbar{x}}^k) =\binom{k+2d-1}{2d}\overbar{x}^{2d},
\]
     where
    $\overbar{x}=(x_1+\dots +x_k)/k$ and the equality holds if and only if $x_1=\dots=x_k$.
\end{lemma}

\begin{proof}
  Let $Z_i$, $i=1,\dots,k$, be i.i.d.~exponential random variables with rate
  parameter $\lambda=1$. We shall use the well-known fact that
  $\EE[Z_i^{t}]=t!$. Let $S_{2d}(x_1,\dots,x_k):=(\sum_{i=1}^{k}
  x_iZ_i)^{2d}/(2d)!$. Then
\begin{equation}\label{eq:E[S]}
    \EE\big[S_{2d}(x_1,\dots,x_k)\big] = \EE\left[\sum_{\ell_1+\dots+\ell_k=2d}
    \prod_{i=1}^{k}\frac{x_i^{\ell_i}Z_i^{\ell_i}}{\ell_i!}\right] = h_{2d}(x_1,\dots,x_k).
\end{equation}

Let $S_{2d}^{+j}(x_1,\dots,x_{n}):=(\sum_{i=1}^{k} x_{i+j}Z_i)^{2d}/(2d)!$ be
the function obtained by a cyclic permutation of the variables in $S_{2d}$,
where the addition in the index of $x_{i+j}$ is taken modulo $k$. As 
\eqref{eq:E[S]} is symmetric in $x_1,\dots,x_n$,
we have $\EE[S_{2d}^{+j}]=\EE[S_{2d}]=h_{2d}$. By convexity of the
function $x \mapsto x^{2d}$,
\begin{align*}
   \frac{(2d)!}{k}\sum_{j=1}^{k} S_d^{+j}(x_1,\dots,x_k) 
   &= \frac{1}{k}\sum_{j=1}^{k} \left(\sum_{i=1}^{k} x_{i+j}Z_i\right)^{2d}\\
   &\geq \left(\frac{1}{k}\sum_{j=1}^{k}\sum_{i=1}^{k}x_{i+j}Z_i \right)^{2d}
   =\left(\sum_{i=1}^{k}\overbar{x}Z_i\right)^{2d} 
   = (2d)!S_{2d}(\overbrace{\overbar{x},\dots,\overbar{x}}^k).
\end{align*}
Taking expectation on both sides then concludes the proof.
\end{proof}

We are now ready to prove~\Cref{thm:main_ineq}.
\begin{proof}[Proof of~\Cref{thm:main_ineq}]
Let $\overbar{X}:=\frac{1}{k}\sum_{i=1}^{k}X_i$, where $k=m-2d+1$ and $X_i$'s are i.i.d.~copies of $X_U$ in~\Cref{lem:spec}. Then by
\Cref{lem:rv_interpret,lem:Schur},
\begin{align*}
q_{m,d}(U) = \EE[h_{2d}(X_1,\dots,X_{m-2d+1})] \geq
\EE\left[\binom{m}{2d}\overbar{X}^{2d}\right].
\end{align*}
By Jensen's inequality and
the fact $\EE[\overbar{X}]=\EE[X_i]=\EE[X_U]=t_{K_2}(U)$ from \Cref{lem:spec},
\begin{align*}
\EE\left[\binom{m}{2d}\overbar{X}^{2d}\right] 
\ge \binom{m}{2d} \EE[\overbar{X}]^{2d}
 =\binom{m}{2d} t_{K_2}(U)^{2d}.
\end{align*}
Combining the two inequalities then completes the proof.
\end{proof}

Without relying on~\Cref{thm:main}, one may also directly prove~\Cref{cor:path} by using the nonnegativity of complete homogeneous symmetric polynomials.
Namely, 
\begin{equation}\label{eq:hunter}
q_{m,d}(U)=\sum_{F\in\E_{2d}^+(P_m)} t_F(U) \geq 0
\end{equation}
for each kernel $U$ and $1\leq d\leq m/2$, a weaker inequality than~\Cref{thm:main_ineq}, is enough. 
The global nonnegativity of $h_{2d}(x_1,\dots,x_k)$, a classical result of Hunter~\cite{Hunter77} and also an easy consequence of~\Cref{lem:Schur}, together with~\Cref{lem:rv_interpret} therefore proves~\Cref{cor:path} directly.

\section{Stability}\label{sec:stability}

One advantage of our proofs in the previous sections is that they also give a stability analysis,
which has not been known for odd cycles $H$ other than the triangle in Sidorenko's theorem.
Roughly speaking, if the inequality in~\Cref{thm:main} is `close' to be an equality, then the graphon $W$ must be `almost' regular.
We begin by showing a stability result corresponding to \Cref{cor:path}.

\begin{theorem}\label{thm:st1}
Let $H$ be a path  with at least $2$ edges or a cycle and let $W$ be a kernel.
For any $\varepsilon\ge 0$, if 
\begin{equation}\label{eq:1+e}
t_{H}(W)+t_{H}(1-W)\leq 2^{1-e(H)}(1+\varepsilon),
\end{equation}
then
\begin{equation}\label{eq:P2}
t_{P_2}(2W-1)\leq \frac{\varepsilon}{e(H)-1}.
\end{equation}
\end{theorem}
\begin{proof}
Let $U=2W-1$.
First, suppose $H=P_m$. By \eqref{eq:tH} and \eqref{eq:hunter},
the assumption \eqref{eq:1+e} implies
\begin{align*}
\varepsilon \ge \sum_{F\in \E^+(P_m)} t_F(U) = \sum_{d=1}^{\lfloor m/2\rfloor}q_{m,d}(U)
\ge q_{m,1}(U)\geq (m-1)t_{P_2}(U),
\end{align*}
which gives \eqref{eq:P2}.
Indeed, $q_{m,1}(U)=(m-1)t_{P_2}(U)+\binom{m-1}{2}t_{K_2}(U)^2$ proves the last inequality.

Suppose now that $H=C_m$. 
By \eqref{eq:tH}, \eqref{eq:hunter},
and the argument for \eqref{eq:E+},
the assumption \eqref{eq:1+e} implies
\begin{align}\label{eq:stability}
\varepsilon \ge 
\sum_{F\in \E^+(C_m)} t_F(U) = \tau + \sum_{e\in E(C_m)}
\sum_{d=1}^{\lfloor(m-1)/2\rfloor} \frac{q_{m-1,d}(U)}{m-2d} 
\geq \frac{m q_{m-1,1}(U)}{m-2} \geq m\cdot t_{P_2}(U),
\end{align}
where $\tau=t_{C_m}(U)$ if $m$ is even and $\tau=0$ otherwise.
This proves \eqref{eq:P2}.
\end{proof}

\Cref{thm:st1} concludes that $t_{P_2}(2W-1)$ is `small' whenever the inequality in \Cref{cor:path} is close to be an equality.
To elaborate on the meaning of $t_{P_2}(2W-1)$ being small,
suppose that $W$ is the indicator graphon of an $n$-vertex graph $G$ and recall that $U=2W-1$.
As
\begin{align*}
 t_{P_2}(U) = \int_{[0,1]^3} U(x,y)U(y,z) dxdydz = \int_0^1 d_U(y)^2 dy,
\end{align*}
where $d_U(y):=\int_0^1 U(x,y) dx$, the inequality $t_{P_2}(U)\leq \varepsilon$ together with Markov's inequality gives
\begin{align*}
   \sqrt{\varepsilon}\cdot \PP \big[d_U(y)^2\geq \sqrt{\varepsilon}\big] \leq \int_0^1 d_U(y)^2 dy \leq \varepsilon.
\end{align*}
That is, 
all but $\sqrt{\varepsilon}n$ vertices in $G$ have degree
between $(1-\varepsilon^{1/4})n/2$ and $(1+\varepsilon^{1/4})n/2$.

If $H=C_m$ with $m$ even in \Cref{thm:st1}, the conclusion becomes even stronger.
Namely, instead of the lower bound $m\cdot t_{P_2}(U)$ in~\eqref{eq:stability},
one may use $\tau=t_{C_m}(U)$ to simply obtain $t_{C_{m}}(U)\leq \varepsilon$.
It is well-known, e.g.,~\cite{CGW89,L12}, that this implies $\|U\|_{\square}\leq \varepsilon^{1/m}$, where $\|.\|_{\square}$ is the \emph{cut norm}.
For the other cases, one cannot expect such a result, as the inequality in~\Cref{thm:main} attains the equality whenever $W$ is a `regular' graphon with density $1/2$, i.e., $d_W(x)=1/2$ almost everywhere. 

\medskip

An analogous stability result for~\Cref{thm:main} can also be obtained. 
\begin{theorem}\label{thm:st2}
Let $H$ be a path with at least $2$ edges or a cycle and let $W$ be a kernel.
For any $\varepsilon\ge 0$, if 
\begin{equation}\label{eq:2e}
t_{H}(W) +t_{H}(1-W) \le t_{K_2}(W)^{e(H)} +t_{K_2}(1-W)^{e(H)}+
2^{1-e(H)}\varepsilon.
\end{equation}
then for $U=2W-1$,
\begin{equation}\label{eq:P2K2}
   t_{P_2}(U)\leq t_{K_2}(U)^2 + \frac{\varepsilon}{e(H)-1}.
\end{equation}
\end{theorem}

\begin{proof}
Suppose that $H=C_m$. 
By \eqref{eq:tH}, \eqref{eq:tK}, \eqref{eq:E+}, and \Cref{thm:main_ineq},
the assumption \eqref{eq:2e} implies
\begin{align*}
\varepsilon \ge  \sum_{F\in \E^+(H)} \big(t_F(U)-t_{K_2}(U)^{e(F)}\big)
   &\geq \frac{m}{m-2}\left(q_{m-1,1}(U) - \binom{m-1}{2}t_{K_2}(U)^2\right) \\
   &= m\left(t_{P_2}(U) - t_{K_2}(U)^2\right), \notag
\end{align*}
which gives \eqref{eq:P2K2}.
The case $H=P_m$ follows in an analogous way.
\end{proof}

The inequality \eqref{eq:P2K2} again implies that $W$ is `almost' regular with respect to the edge density $t_{K_2}(W)$ instead of $1/2$, as $t_{P_2}(U)-t_{K_2}(U)^2$ translates to the variance of $d_U$.
That is, if $W$ is the indicator graphon of an $n$-vertex graph $G$,
then all but $\sqrt{\varepsilon} n$ vertices of $G$ have degree 
between $(p-\varepsilon^{1/4})n$ and
$(p+\varepsilon^{1/4})n$, where $p=t_{K_2}(W)$.
If $H=C_m$ with $m$ even, then we have a stronger conclusion $\|U-t_{K_2}(U)\|_{\square}\leq \varepsilon^{1/2m}$, i.e., $U$ is $\varepsilon^{1/2m}$-close to be quasirandom.

\vspace{5mm}

\noindent\textbf{Acknowledgements.} 
The first author is supported by
the National Research Foundation of Korea (NRF) grants \#2022R1A2C101100911 and \#2016R1A5A1008055.
The second author is supported by the NRF grant  \#2022R1C1C1010300, by Samsung STF Grant SSTF-BA2201-02, and by IBS-R029-C4.
The authors are grateful to David Conlon and Jan Volec for helpful discussions,
to Jozef Skokan for bringing~\cite{BMN22} to their attention,
and to Apoorva Khare for providing references relevant to Schur convexity.

\bibliographystyle{plainurl}
\bibliography{references}

\begin{thebibliography}{10}

\bibitem{Barvinok2005}
Alexander Barvinok.
\newblock Low rank approximations of symmetric polynomials and asymptotic
  counting of contingency tables.
\newblock arXiv:0503170.

\bibitem{BMN22}
Natalie Behague, Natasha Morrison, and Jonathan~A. Noel.
\newblock Common pairs of graphs.
\newblock arXiv:2208.02045.

\bibitem{BR65}
George~R. Blakley and Prabir Roy.
\newblock H\"{o}lder type inequality for symmetrical matrices with non-negative
  entries.
\newblock {\em Proc. Amer. Math. Soc.}, 16:1244--1245, 1965.

\bibitem{BR20path}
Grigoriy Blekherman and Annie Raymond.
\newblock Proof of the {E}rd{\H{o}}s--{S}imonovits conjecture on walks.
\newblock arXiv:2009.10845.

\bibitem{BR80}
Stefan~A. Burr and Vera Rosta.
\newblock On the {R}amsey multiplicities of graphs—problems and recent
  results.
\newblock {\em J. Graph Theory}, 4(4):347--361, 1980.
\newblock \href {https://doi.org/10.1002/jgt.3190040403}
  {\path{doi:10.1002/jgt.3190040403}}.

\bibitem{CGW89}
Fan R.~K. Chung, Ronald~L. Graham, and Richard~M. Wilson.
\newblock Quasi-random graphs.
\newblock {\em Combinatorica}, 9:345--362, 1989.

\bibitem{CL16}
David Conlon and Joonkyung Lee.
\newblock Finite reflection groups and graph norms.
\newblock {\em Adv. Math.}, 315:130--165, 2017.
\newblock \href {https://doi.org/10.1016/j.aim.2017.05.009}
  {\path{doi:10.1016/j.aim.2017.05.009}}.

\bibitem{CL20}
David Conlon and Joonkyung Lee.
\newblock Sidorenko's conjecture for blow-ups.
\newblock {\em Discrete Anal.}, 2:1--14, 2021.
\newblock \href {https://doi.org/10.19086/da} {\path{doi:10.19086/da}}.

\bibitem{E62common}
Paul Erd\H{o}s.
\newblock On the number of complete subgraphs contained in certain graphs.
\newblock {\em Magyar Tud. Akad. Mat. Kutat\'{o} Int. K\"{o}zl}, 7:459--464,
  1962.

\bibitem{ES82_compact}
Paul Erd{\H{o}}s and Miklos Simonovits.
\newblock Compactness results in extremal graph theory.
\newblock {\em Combinatorica}, 2(3):275--288, 1982.

\bibitem{F08}
Jacob Fox.
\newblock There exist graphs with super-exponential {R}amsey multiplicity
  constant.
\newblock {\em J. Graph Theory}, 57(2):89--98, 2008.
\newblock \href {https://doi.org/10.1002/jgt.20256}
  {\path{doi:10.1002/jgt.20256}}.

\bibitem{G59}
Adolph~W. Goodman.
\newblock On sets of acquaintances and strangers at any party.
\newblock {\em Amer. Math. Monthly}, 66:778--783, 1959.
\newblock \href {https://doi.org/10.2307/2310464} {\path{doi:10.2307/2310464}}.

\bibitem{GLLV22}
Andrzej Grzesik, Joonkyung Lee, Bernard Lidick{\'{y}}, and Jan Volec.
\newblock On tripartite common graphs.
\newblock to appear in Combin. Probab. Comput.
\newblock \href {https://doi.org/10.1017/S0963548322000074}
  {\path{doi:10.1017/S0963548322000074}}.

\bibitem{H09}
Hamed Hatami.
\newblock {\em On Generalizations of Gowers Norms}.
\newblock PhD thesis, University of Toronto, 2009.

\bibitem{HHKNR12}
Hamed Hatami, Jan Hladk{\'y}, Serguei Norine, Alexander Razborov, and Dan
  Kr\'{a}l'.
\newblock Non-three-colourable common graphs exist.
\newblock {\em Combin. Probab. Comput.}, 21(5):734--742, 2012.
\newblock \href {https://doi.org/10.1017/S0963548312000107}
  {\path{doi:10.1017/S0963548312000107}}.

\bibitem{H10}
Hatami Hatami.
\newblock Graph norms and {S}idorenko's conjecture.
\newblock {\em Israel J. Math.}, 175(1):125--150, 2010.
\newblock URL: \url{http://dx.doi.org/10.1007/s11856-010-0005-1}, \href
  {https://doi.org/10.1007/s11856-010-0005-1}
  {\path{doi:10.1007/s11856-010-0005-1}}.

\bibitem{Hunter77}
David~B. Hunter.
\newblock The positive-definiteness of the complete symmetric functions of even
  order.
\newblock {\em Math. Proc. Cam. Phil. Soc.}, 82(2):255--258, 1977.

\bibitem{JST96}
Chris Jagger, Pavel {\v{S}}\v{t}ov{\'\i}{\v{c}}ek, and Andrew Thomason.
\newblock Multiplicities of subgraphs.
\newblock {\em Combinatorica}, 16(1):123--141, 1996.
\newblock \href {https://doi.org/10.1007/BF01300130}
  {\path{doi:10.1007/BF01300130}}.

\bibitem{AT21}
Apoorva Khare and Terence Tao.
\newblock On the sign patterns of entrywise positivity preservers in fixed
  dimension.
\newblock {\em Amer. J. Math.}, 143(6):1863--1929, 2021.
\newblock \href {https://doi.org/10.1353/ajm.2021.0049}
  {\path{doi:10.1353/ajm.2021.0049}}.

\bibitem{KVW22}
Dan Kr\'a\v{l}, Jan Volec, and Fan Wei.
\newblock Common graphs with arbitrary chromatic number.
\newblock arXiv:2206.05800.

\bibitem{LSch21}
Joonkyung Lee and Bjarne Sch{\"u}lke.
\newblock Convex graphon parameters and graph norms.
\newblock {\em Israel J. Math.}, 242(2):549--563, 2021.

\bibitem{LS21}
Joonkyung Lee and Alexander Sidorenko.
\newblock On graph norms for complex-valued functions.
\newblock To appear in J. London Math. Soc.
\newblock \href {https://doi.org/https://doi.org/10.1112/jlms.12604}
  {\path{doi:https://doi.org/10.1112/jlms.12604}}.

\bibitem{L66}
David London.
\newblock Two inequalities in nonnegative symmetric matrices.
\newblock {\em Pacific J. Math.}, 16(3):515--536, 1966.

\bibitem{L12}
L{\'a}szl{\'o} Lov{\'a}sz.
\newblock {\em Large Networks and Graph Limits}.
\newblock Amer. Math. Soc. Colloq. Publ. American Mathematical Society, 2012.
\newblock URL: \url{https://books.google.co.uk/books?id=FsFqHLid8sAC}.

\bibitem{Marshall2011}
Albert~W. Marshall, Ingram Olkin, and Barry~C. Arnold.
\newblock {\em Inequalities: theory of majorization and its applications}.
\newblock Springer Series in Statistics. Springer, New York, second edition,
  2011.
\newblock \href {https://doi.org/10.1007/978-0-387-68276-1}
  {\path{doi:10.1007/978-0-387-68276-1}}.

\bibitem{MS59}
H.~P. Mulholland and Cedric A.~B. Smith.
\newblock An inequality arising in genetical theory.
\newblock {\em Amer. Math. Monthly}, 66, 1959.

\bibitem{Sid89common}
Alexander Sidorenko.
\newblock Cycles in graphs and functional inequalities.
\newblock {\em Math. Notes}, 46(5):877--882, 1989.
\newblock \href {https://doi.org/10.1007/BF01139620}
  {\path{doi:10.1007/BF01139620}}.

\bibitem{Sid93}
Alexander Sidorenko.
\newblock A correlation inequality for bipartite graphs.
\newblock {\em Graphs Combin.}, 9(2-4):201--204, 1993.
\newblock URL: \url{http://dx.doi.org/10.1007/BF02988307}, \href
  {https://doi.org/10.1007/BF02988307} {\path{doi:10.1007/BF02988307}}.

\bibitem{Sid92}
Alexander Sidorenko.
\newblock Inequalities for functionals generated by bipartite graphs.
\newblock {\em Discrete Math. Appl.}, 2:489--504, 1993.
\newblock \href {https://doi.org/10.1515/dma.1992.2.5.489}
  {\path{doi:10.1515/dma.1992.2.5.489}}.

\bibitem{Sid96}
Alexander Sidorenko.
\newblock Randomness friendly graphs.
\newblock {\em Random Structures Algorithms}, 8(3):229--241, 1996.
\newblock \href
  {https://doi.org/10.1002/(SICI)1098-2418(199605)8:3<229::AID-RSA6>3.3.CO;2-F}
  {\path{doi:10.1002/(SICI)1098-2418(199605)8:3<229::AID-RSA6>3.3.CO;2-F}}.

\bibitem{Sz78}
Endre Szemer{\'e}di.
\newblock Regular partitions of graphs.
\newblock In {\em Probl\`emes combinatoires et th\'eorie des graphes ({C}olloq.
  {I}nternat. {CNRS}, {U}niv. {O}rsay, {O}rsay, 1976)}, volume 260 of {\em
  Colloq. Internat. CNRS}, pages 399--401. CNRS, Paris, 1978.

\bibitem{TaoSchur}
{T}erence~{T}ao's blog.
\newblock Schur convexity and positive definiteness of the even degree complete
  homogeneous symmetric polynomials.
\newblock URL: \url{https://terrytao.wordpress.com/2017/08/06/}.

\bibitem{Thom89}
Andrew Thomason.
\newblock A disproof of a conjecture of {E}rd{\H{o}}s in {R}amsey theory.
\newblock {\em J. London Math. Soc.}, 2(2):246--255, 1989.
\newblock \href {https://doi.org/10.1112/jlms/s2-39.2.246}
  {\path{doi:10.1112/jlms/s2-39.2.246}}.

\end{thebibliography}

\end{document}